\documentclass[a4paper]{amsart}
\usepackage{amsmath,amsthm,amssymb,amsfonts,mathrsfs,mathtools,color,hyperref, mathtools,crop, graphicx, enumitem}
\usepackage[top=3.5cm, bottom=3cm, left=2.6cm, right = 2.6cm, marginparwidth = 2cm]{geometry}
\usepackage{todonotes}
\usepackage{color}

\theoremstyle{plain}
\begingroup

\newtheorem{theorem}{Theorem}[section]
\newtheorem*{theorem*}{Theorem}
\newtheorem{corollary}[theorem]{Corollary}

\newtheorem{lemma}[theorem]{Lemma}
\endgroup

\theoremstyle{definition}
\begingroup

\endgroup

\theoremstyle{remark}
\begingroup
\newtheorem{remark}[theorem]{Remark}
\newtheorem{example}[theorem]{Example}
\endgroup 

\numberwithin{equation}{section}
\newcommand{\R}{\mathbb R}

\renewcommand{\div}{\operatorname{div}}
\newcommand{\wto}{\stackrel*\rightharpoonup}
\newcommand{\ol}{\overline}
\newcommand{\E}{{\mathcal E}}
\newcommand{\W}{{\mathcal W}}

\renewcommand{\L}{{\mathcal L}}

\newcommand{\F}{{\mathcal F}}

\renewcommand{\H}{{\mathcal H}}
\newcommand{\spt}{\operatorname{spt}}
\newcommand{\inv}{^{-1}}
\newcommand{\cc}{\Subset}

\newcommand{\eps}{\varepsilon}
\newcommand{\id}{\mathrm{id}}

\newcommand{\dx}{\,\mathrm{d}x}
\renewcommand{\d}{\,\mathrm{d}}
 \newcommand{\dy}{\,\mathrm{d}y}
  \newcommand{\dz}{\,\mathrm{d}z}
  \newcommand{\supp}{\mathrm{supp}}

\begin{document}

\title[On the Boundary Regularity of Phase-Fields]{On the Boundary Regularity of Phase-Fields for Willmore's Energy}

\author{Patrick W.~Dondl}
\address{Patrick W.~Dondl\\
Abteilung f\"ur Angewandte Mathematik\\
Albert-Ludwigs-Universit\"at Freiburg\\
Hermann-Herder-Str.~10\\
79104 Freiburg i.~Br.\\
Germany \\
Phone: +49 761 203-5642\\
Fax: +49 761 203-5644}
\email{patrick.dondl@mathematik.uni-freiburg.de}

\author{Stephan Wojtowytsch}
\address{Stephan Wojtowytsch\\Department of Mathematical Sciences\\Durham University\\Durham DH1\,1PT, United Kingdom}
\email{s.j.wojtowytsch@durham.ac.uk}

\date{\today}

\subjclass[2010]{35J67; 49Q20; 49Q10; 49N60; 35J15; 35J25}
\keywords{Willmore energy, phase field, boundary regularity}

\begin{abstract}
We demonstrate that Radon measures which arise as the limit of the Modica-Mortola measures associated to phase-fields with uniformly bounded diffuse area and Willmore energy may be singular at the boundary of a domain and discuss implications for practical applications. We furthermore give partial regularity results for the phase-fields $u_\eps$ at the boundary in terms of boundary conditions and counterexamples without boundary conditions.
\end{abstract}

\maketitle

\section{Introduction}

Phase-field approximations provide a convenient way of treating curvature energies numerically. Typically, the phase-field problem is more stable numerically than the potentially highly non-linear original problem. A classical example of a curvature energy is the Willmore functional 
\[
\W(\Sigma) = \int_\Sigma H^2\d\H^{n-1}
\]
where $\Sigma\subset\R^{n}$ is a hypersurface, $H$ denotes its mean curvature and $\H^k$ the $k$-dimensional Hausdorff measure. The same functional on plane curves is also sometimes referred to as Euler's elastica. 

There are several distinct phase-field approximations of Willmore's energy \cite{bretin2013phase}. The model we will use in the following is due to Bellettini and Paolini \cite{bellettini:1993vg}, based on a functional proposed by De Giorgi \cite[Conjecture 4]{degiorgi:1991jc}. 

Let $\Omega\cc\R^n$ and $W$ be the double-well potential $W(u) = \frac14\,(u^2-1)^2$. Then we consider the Modica-Mortola energy \cite{modica:1987us,MR0445362}
\begin{equation*}
 S_\eps\colon L^1(\Omega)\to \R, \quad S_\eps(u) = \begin{cases}\frac1{c_0}\int_\Omega \frac\eps2\, |\nabla u|^2 + \frac1\eps\,W(u)\dx &u\in W^{1,2}(\Omega)\\ +\infty&\text{else}\end{cases}
\end{equation*}
as an approximation of the perimeter functional and 
\begin{equation*}
\W_\eps\colon L^1(\Omega)\to\R, \quad \W_\eps(u) = \begin{cases}\frac1{c_0\,\eps}\int_\Omega\,\left(\eps\,\Delta u - \frac1\eps\,W'(u)\right)^2\dx&u\in W^{2,2}(\Omega)\\ +\infty &\text{else}\end{cases}
\end{equation*}
as an approximation of Willmore's energy, where $c_0 = \int_{-1}^1\sqrt{2\,W(s)}\:ds = 2\sqrt{2}/3$ is a normalising constant. As proved in \cite{roger:2006ta}, the sum of the functionals satisfies
\begin{equation*}
\left[\Gamma(L^1(\Omega))-\lim_{\eps\to 0}\,(\W_\eps + \Lambda\,S_\eps)\right]\,(\chi_E - \chi_{\Omega\setminus E})\: =\: \W(\partial E) + \Lambda\,\H^{n-1}(\partial E)
\end{equation*}
for any $\Lambda>0$ if $E\cc \Omega$ and $\partial E\in C^2$ in low dimension $n=2,3$. Consider a general sequence $u_\eps$ such that
\[
\limsup_{\eps\to 0}(S_\eps + \W_\eps)(u_\eps) < \infty.
\]
Then the diffuse area measures 
\[
\mu_\eps := \frac1{c_0}\left(\frac\eps2 \,|\nabla u_\eps|^2 + \frac1\eps \,W(u_\eps)\right)\cdot\L^n
\]
which localise the diffuse perimeter functional $S_\eps$ and the diffuse Willmore measures
\[
\alpha_\eps := \frac1{c_0\,\eps }\left(\eps\,\Delta u_\eps - \frac{W'(u_\eps)}\eps\right)\cdot\L^n
\]
which localise the functionals $\W_\eps$ have weak limits $\mu$ and $\alpha$ in the sense of Radon measures, at least for a suitable subsequence. Due to \cite{roger:2006ta}, $\mu$ is the mass measure of an integral $(n-1)$-varifold $V$ in $\Omega$ with square integrable mean curvature and 
\begin{equation}\label{eq willmore alpha}
|H_\mu|^2\cdot\mu \leq \alpha.
\end{equation}
In this article, we will show among other things that the relationship \eqref{eq willmore alpha} is only valid {\em inside} $\Omega$ and that $\mu$ may be very irregular on $\partial\Omega$ if the boundary values of the phase-fields $u_\eps$ are not controlled. The choice of boundary values corresponds to a modelling assumption. In \cite{MR3590663}, we have investigated thin elastic structures in a bounded container, where the natural boundary condition is
\begin{equation}\label{eq strict boundary}
u_\eps \equiv -1, \quad\partial_\nu u_\eps \equiv 0\quad\text{on }\partial\Omega\qquad\text{ or in simpler terms}\quad u_\eps \in -1 + W^{2,2}_0(\Omega)
\end{equation}
to express that the structures are confined to $\Omega$ and only touch the boundary tangentially. Another interesting boundary condition is 
\begin{equation}\label{eq neumann boundary}
\partial_\nu u_\eps \equiv 0 \quad\text{ on }\partial \Omega
\end{equation}
which expresses that the level sets of $u_\eps$ can only meet $\partial\Omega$ at a right angle. This approximates the minimisation problem explored in \cite{alessandroni2014local}. Another possible boundary condition is 
\begin{equation}\label{eq weak boundary}
u_\eps \equiv 1 \text{ on }\Gamma_+, \qquad u_\eps \equiv -1 \text{ on } \Gamma_-, \qquad u_\eps \text{ free on }\partial\Omega\setminus \Gamma_+ \cup \Gamma_-
\end{equation}
which prescribes a phase transition inside $\Omega$ but leaves the particular nature of the transition free. It is clear that any regularity result for $\mu$ or the functions $u_\eps$ inside $\Omega$ can be extended to $\ol \Omega$ under the boundary conditions \eqref{eq strict boundary}, since $u_\eps$ can be extended to the whole space $\R^n$ as a constant function without changing the energy
\[
\E_\eps(u_\eps) := (\W_\eps + S_\eps)(u_\eps).
\]

On the other hand, the regularity of $u_\eps$ and $\mu$ under the boundary values \eqref{eq neumann boundary} or \eqref{eq weak boundary} is less obvious. Furthermore, not specifying boundary values can simplify proofs significantly when local results are considered, see for example \cite[Corollary 2.15]{DW_conv}. In this article, we extend regularity results for the phase-fields $u_\eps$ from \cite{DW_conv,MR3590663}. Our main results are the following.

\begin{theorem}\label{thm:main1}
Let $\Omega \cc\R^n$ for $n=2,3$. Then the following hold true.

\begin{enumerate}
\item Assume that $u_\eps \in C^0(\ol\Omega)$ is uniformly bounded in $L^\infty(\partial\Omega)$. Then $u_\eps$ is uniformly bounded in $L^\infty(\Omega)$ if $n=2$ and in $L^p(\Omega)$ for all $p<\infty$ if $n=3$.

\item Assume that $\partial\Omega\in C^2$ and $\partial_\nu u_\eps \equiv 0$ on $\partial\Omega$ for all $\eps>0$. Then $u_\eps$ is uniformly bounded in $L^\infty(\Omega)$ and 
\[
|u_\eps(x) - u_\eps(y)| \leq \frac{C}{\eps^\gamma}\,|x-y|^\gamma \qquad \forall\ x\in \Omega, y\in B_\eps(x)\cap \Omega
\]
with $\gamma <1$ if $n=2$ and $\gamma\leq 1/2$ if $n=3$. The constant $C$ depends on $n,\gamma, \Omega$ and $\limsup_{\eps\to 0}\E_\eps(u_\eps)$.

\item If either condition is given and $u_\eps\to u$ in $L^1(\Omega)$, then $u_\eps\to u$ in $L^p(\Omega)$ for all $1\leq p <\infty$.
\end{enumerate}
\end{theorem}

Further results can be found in the main text. The proof is split over Lemmas \ref{third regularity lemma}, \ref{fourth regularity lemma} and \ref{second regularity lemma}. On the other hand, we have the following results on situations where phase fields fail to be regular at the boundary.

\begin{theorem} \label{thm:main2}
Let $\partial\Omega\in C^2$. Then the following hold true.

\begin{enumerate}
\item There exists a sequence $u_\eps\in W^{2,2}(\Omega)$ such that $(\W_\eps + S_\eps)(u_\eps)\to 0$, but $u_\eps$ is not bounded in $L^\infty(\Omega)$.

\item There exists a sequence $u_\eps$ such that such that $\alpha = 0$, $\mu = 0$ but the Hausdorff limit
\[
K\coloneqq \lim_{\eps\to 0} u_\eps^{-1}(I)\qquad \emptyset \neq I \cc (-1,1)
\]
of level sets or their unions contains an open subset of $\partial\Omega$. Similar constructions give $K = \{x_0\}$ or $K=\gamma$ for a point $x_0\in\partial\Omega$ and a closed curve $\gamma\subset\partial\Omega$.

\item Let $S>0$ and $\emptyset \neq I \cc (-1,1)$. Then there exists a point $x_0\in \partial\Omega$ and a sequence $u_\eps\in W^{2,2}(\Omega)$ such that $|u_\eps|\leq 1$ in $\overline\Omega$, $\W_\eps(u_\eps) \equiv 0$, $\mu_\eps(\Omega)\equiv S$, $K=\emptyset$ and $\mu = S\cdot\delta_{x_0}$. 
\end{enumerate}

If $\Omega$ is convex, any point $x_0$ or closed curve $\gamma$ in $\partial\Omega$ can be chosen and $u_\eps$ may be such that it is not uniformly bounded in $\Omega\cap U$ for all open sets $U$ with $U\cap \partial\Omega\neq\emptyset$.
\end{theorem}

This shows that for example the minimisation problem for
\[
\F_\eps = \W_\eps + \eps^{-\sigma}(S_\eps -S)^2
\]
is not physically meaningful without boundary conditions or with partly free boundary conditions \eqref{eq weak boundary} if $\partial_{\text{free}}\Omega := \partial\Omega\setminus (\Gamma_+\cup \Gamma_-) \neq \emptyset$. A minimising sequence is given by the superposition of a phase-field making an optimal transition along a minimal surface spanning a suitable boundary curve inside $\partial_{\text{free}}\Omega$ and a second phase-field creating an atom of $\mu$ of the correct size at a single point $x\in \partial_{\text{free}}\Omega$. This can be realised with energy $\W_\eps(u_\eps) \to 0$ as $\eps\to 0$. 

The question under which boundary conditions other than \eqref{eq strict boundary} the measure $\mu$ can be expected to be regular at the boundary for either finite energy sequences or minimising sequences remains open.

\section{Positive Results on Boundary Regularity}

In this chapter, we describe partial regularity results for weakly controlled boundary values.

\begin{lemma}\label{third regularity lemma}
Assume that $u_\eps$ is continuous on $\ol \Omega$ and there is $\theta\geq 1$ such that $|u_\eps|\leq\theta$ on $\partial\Omega$ for all $\eps>0$. Then the following hold true.
 
\begin{enumerate}
 \item There exists $C>0$ such that $\mu_\eps(\{|u_\eps|\geq \theta\})\leq C\,\eps^2$.
 
 \item For the set $\tilde\Omega_\eps = \{x\in \Omega\:|\: B_{2\eps}(x)\subset\Omega\}$ we can show that there exists $C$ depending only on $\bar\alpha, \gamma$ and $\theta$ such that
\[
||u_\eps||_{\infty,\tilde\Omega_\eps}\leq C, \qquad |u_\eps(y) - u_\eps(z)|\leq \frac{C_{\bar\alpha, \theta,\gamma}}{\eps^{\gamma}}\,|y-z|^{\gamma}
\]
if there is $x\in\tilde\Omega_\eps$ such that $y, z\in B_\eps(x)$ and $\gamma\leq 1/2$ if $n=3$, $\gamma<1$ if $n=2$.
\end{enumerate}
\end{lemma}

\begin{proof}
This proof is an adaptation of the proof of Lemma \cite[Lemma 3.1]{MR3590663} using a modified argument in the first step of the proof. We observe that for the proof of Lemma \cite[Lemma 3.1]{MR3590663} to work, we needed that $B_{2\eps}(x)\subset\Omega$ to employ the elliptic inequality
\[
||\tilde u_\eps||_{2,2, B_1(0)}\leq C\,\left(||\tilde u_\eps||_{2,B_2(0)} + ||\Delta \tilde u_\eps||_{2, B_2(0)}\right)
\]
and an estimate of $\int_{B_{2\eps}(x)}\frac1{\eps^n}W'(u_\eps)^2\dx$. The first one we are given directly by the choice of $\Omega_\eps^\beta$ or $\tilde \Omega_\eps$, the second can be obtained through integration by parts
\begin{align*}
c_0\,\alpha_\eps(\{&|u_\eps|>\theta'\}) = \int_{\{|u_\eps|>\theta'\}} \frac1\eps \left(\eps\,\Delta u_\eps - \frac1\eps\,W'(u_\eps)\right)^2\dx\\
	&= - \frac2\eps \int_{\partial\{|u_\eps|>\theta'\}}W'(u_\eps)\,\partial_\nu u_\eps\d\H^{n-1}\\ &\qquad + \int_{\{|u_\eps|>\theta'\}} \eps\,(\Delta u_\eps)^2 + \frac2\eps\,W''(u_\eps)\,|\nabla u_\eps|^2 + \frac1{\eps^3}\,W'(u_\eps)^2\dx\\
	&\geq \int_{\{|u_\eps|>\theta'\}} \eps\,(\Delta u_\eps)^2 + \frac4\eps\,|\nabla u_\eps|^2 + \frac1{\eps^3}\,W'(u_\eps)^2\dx
\end{align*}
for $\theta'>\theta$ when $\{|u_\eps|>\theta'\}$ is a Caccioppoli set (i.e.\ for almost all $\theta'>\theta$). If $|u_\eps|<\theta'$ on $\partial\Omega$, the set $\{u_\eps>\theta'\}$ does not touch the boundary $\partial\Omega$, so $\partial\{u_\eps>\theta'\}\subset \{u_\eps = \theta'\} \subset\Omega$. Because $W'(\theta)>0$ and $\partial_\nu u_\eps$ is inward pointing on $\partial\{u_\eps>\theta\}$, the boundary integral is non-positive. The rest of the argument goes through as before. Additionally, taking $\theta'\to \theta$ establishes the first claim. 
\end{proof}

\begin{remark}
The same bound holds for example on $\widetilde \Omega_{\eps^{1/2}} = \{x\in \Omega\:|\:B_{\eps^{1/2}}(x)\subset \Omega\}$ without boundary values. In that situation, we employ the estimate from \cite[Proposition 3.6]{roger:2006ta} to bound
\[
\frac1{\eps^3}\int_{\{|u_\eps|>1\}} W'(u_\eps)^2\dx \leq C.
\]
\end{remark}

Another situation with a similar improvement is that of prescribed Neumann boundary data.

\begin{lemma}\label{fourth regularity lemma}
Assume that $\Omega$ is a Caccioppoli set and $\partial_\nu u_\eps = 0$ almost everywhere on $\partial \Omega$ with respect to the boundary measure $|D\chi_\Omega|$. Then the following hold true.
 
\begin{enumerate}
 \item There exists $C>0$ such that $\mu_\eps(\{|u_\eps|\geq 1\})\leq C\,\eps^2$.
 
 \item For the set $\tilde\Omega_\eps = \{x\in \Omega\:|\: B_{2\eps}(x)\subset\Omega\}$ we can show that there exists $C$ depending only on $\bar\alpha$ and $\gamma$ such that
\[
||u_\eps||_{\infty,\tilde\Omega_\eps}\leq C, \qquad |u_\eps(y) - u_\eps(z)|\leq \frac{C}{\eps^{\gamma}}\,|y-z|^{\gamma}
\]
if there is $x\in\tilde\Omega_\eps$ such that $y, z\in B_\eps(x)$. Here $\gamma\leq 1/2$ if $n=3$, $\gamma<1$ if $n=2$.
\end{enumerate}

If $\partial\Omega \in C^2$ and $\partial_\nu u_\eps = 0$ almost everywhere on $\partial \Omega$, then the second statement can be sharpened as follows:

\begin{enumerate}
\item[2'.] For all $x\in \ol\Omega$ there exists a constant $C$ depending only on $\bar\alpha,\bar \mu,\gamma$ and $\partial\Omega$ such that
\[
|u_\eps(x)| \leq C, \qquad |u_\eps(y) - u_\eps(z)|\leq \frac{C}{\eps^\gamma}\,|x-y|^\gamma\qquad\forall\ y,z\in B_\eps(x)\cap \ol\Omega.
\]
The dependence of $C$ on $\partial \Omega$ vanishes in the limit $\eps\to 0$.
\end{enumerate}
\end{lemma}

In particular, for regular boundaries, the Neumann condition implies the boundedness of solutions (in particular also on the boundary).

\begin{proof}
As before, we obtain
\begin{align*}
\alpha_\eps(\{&|u_\eps|>\theta'\}) = \int_{\{|u_\eps|>\theta'\}} \frac1\eps \left(\eps\,\Delta u_\eps - \frac1\eps\,W'(u_\eps)\right)^2\dx\\
		&= - \frac2\eps \int_{\partial\Omega\cap \partial\{|u_\eps|>\theta'\}}W'(u_\eps)\,\partial_\nu u_\eps\d\H^{n-1} - \frac2\eps \int_{\partial\{|u_\eps|>\theta'\}\cap \Omega}W'(u_\eps)\,\partial_\nu u_\eps\d\H^{n-1}\\
		&\qquad + \int_{\{|u_\eps|>\theta'\}} \eps\,(\Delta u_\eps)^2 + \frac2\eps\,W''(u_\eps)\,|\nabla u_\eps|^2 + \frac1{\eps^3}\,W'(u_\eps)^2\dx\\
	&\geq \int_{\{|u_\eps|>\theta'\}} \eps\,(\Delta u_\eps)^2 + \frac4\eps\,|\nabla u_\eps|^2 + \frac1{\eps^3}\,W'(u_\eps)^2\dx
\end{align*}
for any $\theta'>1$ such that $\{|u_\eps|>\theta'\}$ is a Caccioppoli set. Here the boundary integral can be split into two parts, one of which has a sign, while the other one vanishes due to the Neumann condition. This implies the boundedness on $\tilde\Omega_\eps$ and the bound on the mass measures $\mu_\eps(\{|u_\eps|>\theta'\})$ as before. We can take $\theta'\to 1$ to prove the first part of the Lemma.

Now assume that $\partial\Omega\in C^2$ and pick $x\in\partial\Omega$. The rest of the argument is a fairly standard `straightening the boundary' argument with the feature that the boundary becomes flatter as $\eps\to 0$. Without loss of generality, we assume that $x=0$. We may now blow up to 
\[
\tilde u_\eps : B_{2}(0)\cap (\Omega/\eps)\to \R, \qquad \tilde u_\eps(y) = u_\eps(\eps y).
\]
We pick a $C^2$-diffeomorphism $\phi_\eps:B_2(0)\to B_2(0)$ such that 
\begin{enumerate}
\item $\phi_\eps (\Omega/\eps \cap B_2(0))= B_2^+(0)$,
\item $\phi_\eps \to \id_{B_2(0)}$ in $C^2(B_2(0), B_2(0))$ as the domain becomes increasingly flat,
\item under $\phi_\eps$, the normal to $\partial\Omega/\eps$ gets mapped to $e_n$ on the boundary, i.e.\ the orthogonality condition is preserved.
\end{enumerate} 
With this we obtain a function 
\[
{\tilde w_\eps}:B_2^+(0)\to\R, \qquad \tilde w_\eps(y) = u_\eps(\phi_\eps^{-1}(y))
\]
in flattened coordinates. Since $\phi_\eps$ is $C^2$-smooth, it preserves $W^{2,2}$-functions and it is easy to calculate
\begin{align*}
\partial_i\tilde u_\eps &= \partial_i (\tilde w_\eps\circ\phi_\eps)\\
	&= \partial_i(\phi_{\eps})_j\,\left((\partial_j\tilde w_\eps) \circ \phi_\eps\right)\\
\partial_{ij}\, \tilde u_\eps &=\partial_{ij}(\phi_\eps)_k\,\left((\partial_k\tilde w_\eps) \circ \phi_\eps\right) + \partial_i(\phi_\eps)_k \,\partial_j(\phi_\eps)_l\,\left((\partial_{kl}\tilde w_\eps)\circ \phi_\eps\right).
\end{align*}
In shorter notation, this means that 
\[
\nabla \tilde u_\eps = D\phi\cdot\nabla \tilde w_\eps, \qquad \Delta \tilde u_\eps = a^{ij}_\eps\,\partial_{ij}\tilde w_\eps + \langle \Delta\phi_\eps,\nabla \tilde w_\eps \rangle 
\]
with
\[
a^{ij}_\eps = \langle \partial_i\phi_\eps,\partial_j\phi_\eps\rangle.
\]
The coefficients are $C^1$-differentiable -- so the associated operator $A_\eps$ can be equivalently written in divergence form -- and $C^1$-close to $\delta_{ij}$. We observe that 
\[
\left(\Delta \tilde u_\eps - W'(\tilde u_\eps)\right)(\phi(y)) = \left(\partial_i\left(a_\eps^{ij}\,\partial_j\tilde w_\eps\right) - (\partial_i\,a_\eps^{ij})\partial_j\tilde w_\eps + \langle \Delta\phi_\eps,\nabla\tilde w_\eps\rangle - W'(\tilde w_\eps) \right)(y).
\]
We extend $\tilde w_\eps$ by even reflection to the whole ball $B_2(0)$, which preserves the $W^{2,2}$-smoothness since we preserved the property that $\partial_\nu\tilde u_\eps =0$ on the boundary when straightening the boundary. We observe that 
\[
\partial_i\,(a_\eps^{ij}\partial_j\tilde w_\eps) - \langle \div A_\eps - \Delta \phi_\eps, \nabla \tilde w_\eps\rangle =: f_\eps \in L^2(B_2(0))
\]
since
\begin{align*}
\int_{B_2(0)}W'(\tilde w_\eps)^2\dy &= 2\int_{B_2^+(0)}W'(\tilde w_\eps)^2(y)\dy\\
	&= 2\int_{\Omega/\eps\cap B_2(0)} W'(\tilde u_\eps((z)))\,\det(D\phi^{-1}_\eps)(z)\dz\\
	&\leq 2(1+c_\eps) \int_{B_{2\eps}(x)}\frac1{\eps^n}W'(\tilde u_\eps)\dz\\
	&\leq C
\end{align*}
as shown above. The constants $c_\eps$ vanish as $\eps\to 0$ and  $\phi_\eps\to \id$. The coefficients $a_{ij}$ are uniformly elliptic and approach $\delta_{ij}$ uniformly as $\eps\to 0$, so we can employ the elliptic estimate
\begin{align*}
||\nabla \tilde w_\eps||_{L^2(B_{3/2})} &\leq C\left\{||\tilde w_\eps||_{L^2(B_2)} + ||f_\eps + \langle \div A_\eps - \Delta \phi_\eps, \nabla \tilde w_\eps\rangle||_{L^2(B_2)}\right\}\\
	&\leq C\left\{||\tilde w_\eps||_{L^2(B_2)} + ||f_\eps||_{L^2(B_2)} + ||\div A_\eps - \Delta \phi_\eps||_{L^\infty(B_2)}\, ||\nabla \tilde w_\eps||_{L^2(B_2)}\right\}.
\end{align*}
The constant is uniform in $\eps$ and $||\div A_\eps - \Delta \phi_\eps||_{L^\infty(B_2)}\to 0$ as $\eps\to 0$, so we can bring the term to the other side and obtain a uniform $W^{1,2}$-bound for all sufficiently small $\eps$, where the necessary smallness depends only on $\W_\eps(u_\eps)$ and $\partial\Omega$. In a second step, this gives us a uniform bound on $||\tilde w_\eps||_{W^{2,2}(B_1(0))}$, which gives us a uniform bound on $||\tilde u_\eps||_{W^{2,2}(B_{3/2}(0)\cap \Omega/\eps)}$ after transforming back. The rest follows by Sobolev embeddings as in \cite[Lemma 3.1]{MR3590663}.
\end{proof}

\begin{remark}
The case that $\Omega$ has finite perimeter and $\partial_\nu u_\eps=0$ almost everywhere on the reduced boundary is a generalisation of the situation in which $\partial\Omega\in C^2$ and the level sets of $u_\eps$ meet $\partial\Omega$ at a ninety degrees angle. Such conditions arise naturally when we search for surfaces of minimal perimeter bounding a prescribed volume and may be useful also for models containing Willmore's energy \cite{alessandroni2014local}.
\end{remark}

We give an improvement of the $L^\infty$-bound up to the boundary which implies $L^p$-convergence for all finite $p$.

\begin{lemma}\label{second regularity lemma}
Assume that there is $\theta\geq 1$ such that $|u_\eps|\leq\theta$ on $\partial\Omega$ for all $\eps>0$. Then the following hold true.

\begin{enumerate}

\item If $n=2$, $\partial\Omega\in C^{1,1}$ and $\theta>1$, then for every $\beta<1$ there exists a constant $C$ depending only on $\bar\alpha, \theta, \Omega$ and $\beta$ such that $\sup_{x\in \Omega} |u_\eps(x)|\leq \theta+ C\eps^\beta$ for all $\eps>0$.

If $\theta =1$, then for every $\beta<1/2$ there exists a constant $C$ depending only on $\bar\alpha, \Omega$ and $\beta$ such that $\sup_{x\in \Omega} |u_\eps(x)|\leq 1+ C\eps^\beta$ for all $\eps>0$.

\item If $n=3$ and $\partial\Omega\in C^{1,1}$, then for every $p<\infty$ there exists $C$ depending only on $\bar\mu,\bar\alpha, \theta, p$ and $\Omega$ such that $||u_\eps||_{p,\Omega}\leq C$. Furthermore, for every $\sigma>0$ there exists $C$ depending only on $\bar\alpha, \theta, \Omega$ and $\sigma$ such that $||u_\eps||_{\infty,\Omega}\leq C\,\eps^{-\sigma}$.
\end{enumerate}
\end{lemma}

We assume that also in three dimensions, uniformly bounded boundary values lead to uniform interior bounds.

\begin{proof}
The proof is a modified version of that of \cite[Proposition 3.6]{roger:2006ta}. We follow that proof closely, but use a different maximum principle.

Let $\theta'>\theta\geq 1$ such that $\{|u_\eps|>\theta'\}$ has finite perimeter and define $w_\eps\coloneqq  (u_\eps- \theta')_+$. Then $w_\eps\in W^{1,2}_0(\Omega)$ and from the same integration by parts as before we obtain that
\[
||w_\eps||_{1,2,\Omega}^2 \leq \int_{\{u_\eps>\theta'\}} W'(u_\eps)^2 + |\nabla u_\eps|^2 \leq \alpha_\eps(\Omega)\,\eps.
\]
The function satisfies 
\begin{align*}
\int_\Omega w_\eps\,(-\Delta\phi)\dx &= \int_{\{u_\eps>\theta'\}} (u_\eps-\theta')\,(-\Delta\phi)\dx \\
	&= -\int_{\partial\{u_\eps>\theta'\}} (u_\eps-\theta')\,\partial_\nu\phi \d\H^{n-1} + \int_{\{u_\eps>\theta'\}}\langle \nabla\phi, \nabla u_\eps\rangle\dx\\
	&= \int_{\partial \{u_\eps>\theta'\}}\phi\,\partial_\nu u_\eps - (u_\eps-\theta')\,\partial_\nu\phi\d\H^{n-1} + \int_{\{u_\eps>\theta'\}} \phi\,(-\Delta u_\eps)\dx\\
	&\leq \int_{\{u_\eps>\theta'\}} \phi\,(-\Delta u_\eps)\dx
\end{align*}
for $\phi\geq 0$. Again, this holds true because $\partial\{u_\eps>\theta'\}\subset\{u_\eps=\theta'\}$. Obviously
\begin{align*}
\int_{\{u_\eps>\theta'\}} \phi\,(-\Delta u_\eps)\dx &= \int_{\{u_\eps>\theta'\}} \left(-\Delta u_\eps + \frac1{\eps^2}W'(u_\eps) - \frac1{\eps^2}\,W'(u_\eps)\right)\,\phi\dx\\
	&\leq \int_{\{u_\eps>\theta'\}}\frac1\eps\,\left(h_\eps - \frac1\eps\,W'(\theta')\right)_+\phi\dx,
\end{align*}
so $-\Delta w_\eps \leq \frac1\eps\,\chi_{\{u_\eps>\theta'\}}\left(h_\eps - \frac1\eps\,W'(\theta')\right)_+ $ in the distributional sense. When we consider the solution $\psi_\eps \in W^{1,2}_0(\Omega)$ of the problem
\[
- \Delta\psi_\eps = \frac1\eps\,\left(h_\eps - \frac1\eps\,W'(\theta')\right)_+\chi_{\{u_\eps>\theta'\}},
\]
the weak maximum principle \cite[Theorem 8.1]{gilbarg:2001vb} applied to $w_\eps - \psi_\eps$ implies that
\begin{equation}\label{equation estimate u psi}
u_\eps \leq \theta + w_\eps \leq \theta + \psi_\eps.
\end{equation}
We proceed to estimate
\begin{align*}
||\Delta\psi_\eps||_{q,\Omega}^q &= \eps^{-q} \int_{\{u_\eps>\theta'\}}\left(h_\eps - \frac1\eps\,W'(\theta')\right)_+^q\dx \\
	& \leq \eps^{-q} \left(\int_{\{u_\eps>\theta'\}} 1\dx\right)^{1-q/2}\left(\int_\Omega\,h_\eps^2\dx\right)^{q/2}\\
	&\leq \eps^{-q} \left(\frac{\eps^3}{W'(\theta')^2}\int_{\{u_\eps>\theta'\}} \frac1{\eps^3}\,W'(u_\eps)^2\dx\right)^{1-q/2}\left(\eps\int_{\Omega}\frac1\eps\,h_\eps^2\dx\right)^{q/2}\\
	&\leq c_{\bar\alpha,q}\,(W'(\theta'))^{q-2}\,\eps^{-q + 3(1-q/2) + q/2}\\
	&=  c_{\bar\alpha,q} \,(W'(\theta'))^{q-2}\,\eps^{3 - 2q}
\end{align*}
for $1\leq q<2$. Thus $||\Delta\psi_\eps||_{q,\Omega}\leq  C_{\bar\alpha,q}\,(W'(\theta'))^{1-2/q}\,\eps^{3/q -2}$, and by the elliptic estimate \cite[Lemma 9.17]{gilbarg:2001vb}, we have
\[
||\psi_\eps||_{2,q,\Omega} \leq c_{\Omega,\bar\alpha,q}\,(W'(\theta'))^{1-2/q}\,\eps^{3/q-2}.
\]
Let us insert this estimate into \eqref{equation estimate u psi}. If $n=3$, we take $q=3/2$ and use that $W^{2,3/2}(\Omega)$ embeds into $L^p(\Omega)$ for all finite $p$. Thus (taking some $\theta'>1$ if $\theta=1$), we see that $u_\eps\leq \theta' + \psi_\eps$ where $\psi_\eps$ is uniformly bounded in $L^p(\Omega)$. We may use the same argument on the negative part of $u_\eps$, so in total $u_\eps$ is uniformly bounded in $L^p(\Omega)$ for all $1\leq p<\infty$ by domination through $\psi_\eps$. Taking $q=3/(2-\sigma)>3/2$ proves the $L^\infty$-estimate by the same comparison.

If $n=2$, we have a Sobolev embedding $W^{2,q} (\Omega) \to L^\infty(\Omega)$ for all $q>1$. Assuming that $\theta>1$ and $\beta<1$ we take $\theta'\to \theta$ to obtain 
\[
u_\eps \leq \theta + w_\eps \leq \theta + \psi_\eps \leq \theta + C_{\Omega,\bar\alpha,q}\,(W'(\theta))^{1-2/q}\, \eps^{3/q-2}.
\]
For $q = 3/(2+\beta)$, this gives $u_\eps\leq 1 + C\,\eps^\beta$. Here $q\in (1,2)$ is admissible since $\beta\in(0,1)$. If $\theta=1$, we may take $0 < \beta<1/2$, $q=(3-2\beta)/2 \in(1,2)$ and $1 + \eps^\beta \leq \theta' \leq 1 + 2\eps^\beta$ to obtain 
\[
|u_\eps|\leq 1 + C_{\Omega, \bar\alpha, q} \eps^{\beta(1-2/q) + (3/q-2)} = 1 + C_{\Omega, \bar\alpha, q}\eps^\beta
\]
with the approximation $W'(\theta') = O(\eps^\beta)$.
\end{proof}

\begin{corollary}
If $u_\eps \to u$ in $L^1(\Omega)$ and either
\begin{enumerate}
\item $u_\eps\in C^0(\ol\Omega)$ and there exists $\theta\geq1$ such that $|u_\eps|\leq\theta$ on $\partial\Omega$ for all $\eps>0$ or
\item $\partial\Omega\in C^2$ and $\partial_\nu u_\eps = 0$ a.e.\ on $\partial\Omega$,
\end{enumerate}
then $u_\eps\to u$ in $L^p(\Omega)$ for all $1\leq p<\infty$.
\end{corollary}

\begin{proof}
The sequence $u_\eps$ converges to $u$ in $L^1(\Omega)$ and is bounded in $L^q(\Omega)$ for all $q<\infty$ (or even $L^\infty(\Omega)$). H\"older's inequality implies $L^p$-convergence.
\end{proof}

\begin{remark}
If $n=2$, $\beta<1/2$ and $|u_\eps|\leq 1+\eps^\beta$ on $\partial\Omega$, then the proof still shows that
\[
\sup_{\Omega} |u_\eps|\leq 1+ C\,\eps^\beta
\]
for this particular $\beta$. The case $\beta = 1/2$ is still open at the boundary.
\end{remark}

For a counterexample to uniform boundedness on $\Omega$ without boundary conditions, see Example \ref{counterexample 1}. Even with boundary values satisfying $|u_\eps|\leq 1$ on $\partial\Omega\in C^2$, we shall construct a sequence $u_\eps$ for which uniform H\"older continuity fails at the boundary in Example \ref{counterexample 2}.

\section{Counterexamples to Boundary Regularity}\label{section counterexamples}

The idea here is simple: namely, the energy $\W_\eps$ can be seen to control the $W^{2,2}$-norm of blow ups of phase-fields onto $\eps$-scale since those are asymptotic to bounded entire solutions of the stationary Allen-Cahn equation $-\Delta \tilde u + W'(\tilde u) = 0$ at (almost all) points away from the boundary. At the boundary on the other hand, the asymptotic behaviour corresponds to solutions of the same equation on half-space, whose behaviour is essentially governed by their boundary values. To make this precise, take $h\in C_c^\infty(\R^n)$ and $H\coloneqq  \{x_n>0\}$. The energy
\[
\F\colon W^{1,2}_{loc}(H)\to \R\cup\{\infty\}, \quad \F(u) = \int_H\frac12\,|\nabla u|^2 + W(u)\dx
\]
has a minimiser $\tilde u$ in the affine space $(1+h) + W^{1,2}_0(H)$ by the direct method of the calculus of variations. Namely, take a sequence $u_k$ such that $\lim_{k\to\infty}\F(u_k) = \inf\F(u)  \leq \F(h+1) < \infty$. Then
\[
||\nabla u_k||_{L^2(H)}\leq C, \qquad\text{and}\quad (u_k-1)^2(x) \leq (u_k-1)^2(u_k+1)^2(x) = 4\,W(u_k(x))
\]
at all points $x\in H$ such that $u_k(x)\geq 0$. Using the boundary values, also the negative part of $u_k$ is uniformly controlled in $L^2(H)$ by the $H^1$-semi norm. Thus the sequence $u_k$ is bounded in $W^{1,2}(H)$ and there exists $\tilde u$ such that $u_k\wto \tilde u$ (up to a subsequence).  Since the affine space is convex and strongly closed, it is weakly = weakly* closed and $\tilde u\in 1+h+W_0^{1,2}(H)$. For any $R>0$, we can use the compact embedding $W^{1,2}(B_R^+)\to L^4(B_R^+)$ to deduce that 
\[
\int_{B_R^+}\frac12\,|\nabla \tilde u|^2 + W(\tilde u)\dx \leq \liminf_{k\to\infty} \int_{B_R^+}\frac12\,|\nabla u_k|^2 + W(u_k)\dx \leq \liminf_{k\to\infty} \int_{H}\frac12\,|\nabla u_k|^2 + W(u_k)\dx.
\]
Letting $R\to\infty$ shows that $\tilde u$ is in fact a minimiser of $\F$. If $h\geq 0$, then
\[
1 + (\tilde u-1)_+\in 1 + h + W^{1,2}_0(H), \qquad \F\left( 1 + (\tilde u-1)_+ \right) \leq \F(\tilde u)
\]
with strict inequality unless $\tilde u = 1 + (\tilde u-1)_+$. Since we assume $\tilde u$ to be a minimiser, we find that $\tilde u\geq 1$ almost everywhere. The same argument shows that $\tilde u\leq 1 + ||h||_\infty$ almost everywhere. Calculating the Euler-Lagrange equation of $\F$, we see that $\tilde u$ is a weak solution of
\[
-\Delta\tilde u + W'(\tilde u) = 0.
\]
On the convex set
\[
C_h\coloneqq  \{ u\in W^{1,2}(H) \:|\: u = 1+h\text{ on }\partial H, u\geq 1\}
\]
the operator
\[
A\colon C_h\to W^{-1,2}(H), \quad A(u) = -\Delta u + W'(u)
\]
is well-defined (since $n\leq 3$ and $W'$ has cubical growth) and strongly monotone, so the equation $Au=0$ has a unique solution $\tilde u\in C_h$ which coincides with the minimiser $\tilde u$ of $\F$ in $1+h+ W_0^{1,2}(H)$). A bootstrapping argument via elliptic regularity theory shows that $\tilde u\in C^\infty_{loc}(\overline H)$. By trace theory we have that
\[
||h||_{2, \partial H}^2 \: =  \: ||\tilde u - 1||_{2,\partial H}^2 \:\leq \: ||\tilde u-1 ||_{1,2,H}^2 /2  \:\leq \:\F(\tilde u)\: \leq \: \F(1+h) .
\]
In this way, we can fully control the mass density $\tilde\mu = \frac12\,|\nabla \tilde u|^2 + W(\tilde u)$ created by $\tilde u$ in terms of its boundary values. For later purposes, we have to obtain suitable decay estimates for the functions $\tilde u$ depending on $h$. In a first step, we show that the limit $\lim_{|x|\to\infty} \tilde u(x) = 1$ exists. Assume the contrary. Then there exist $\theta>1$ and a sequence $x_k\in H$ such that 
\[
|x_k|\to\infty, \qquad \tilde u(x_k) \geq \theta.
\]
Taking a suitable subsequence, we may assume that the balls $B_1(x_k)$ are disjoint and $|x_k|\geq R+2$ is so large that $h$ is supported in $B_R(0)$. If $B_2(x_k)\subset H$, we may proceed as in Lemma \ref{regularity lemma} to deduce uniform H\"older continuity on the balls $B_1(x_k)$ from the $L^\infty$-bound to $\tilde u$ and the fact that $\tilde u$ solves $\Delta\tilde u =  W'(\tilde u)$. This means that there exists $r>0$ such that $\tilde u \geq (1+\theta)/2$ on $B_r(x_k)$. Otherwise, the same argument still goes through after extending $\tilde u$ by a standard reflection principle and the fact that the boundary values are constant on $\partial H\cap B_2(x_k)$. The geometry of $H$ gives us $\L^n(B_r(x_k)\cap H)\geq \omega_n\,r^n/2$. So we deduce that
\[
\F(\tilde u) \geq \sum_{k=0}^\infty \int_{B_r(x_k)} W(\,(1+\theta)/2)\dx \geq \sum_{k=0}^\infty \,W((1+\theta)/2) \,\omega_n\,r^n/2= \infty
\]
in contradiction to the definition of $\tilde u$. Now we can estimate the decay of $\tilde u$ in a more precise fashion. Since $h\in C_c(\partial H)$, there is $C_h>0$ such that $h\leq C_h\,e^{-|x|}$ on $\partial H$. To simplify the following calculations, we assume that $C_h=1$. Then we claim that $1\leq u\leq 1+ e^{-|x|}$ for all $x\in \R^n$. Assume the contrary and observe that $\psi(x) = 1+ e^{-|x|}$ satisfies
\[
\Delta\psi(x) = \left(1 + \frac{1-n}{|x|}\right)\,e^{-|x|},\qquad W'(\psi(x)) = \left(2 + 3\,e^{-|x|} + e^{-2\,|x|}\right)\,e^{-|x|},
\]
so in particular $\Delta \psi(x) \leq W'(\psi(x))$ for all $x\in \R^n$. Since $\tilde u=h \leq \psi$ on $\partial H$ by assumption and $\lim_{|x|\to \infty}\tilde u(x) = 1$, there must be a point $x_0\in H$ such that
\[
(\psi - u)(x_0) = \min_{H} (\psi-u) < 0,
\]
but then
\[
\Delta(\psi - u)(x_0) \leq W'(\psi(x_0)) - W'(u(x_0)) < 0
\]
so $\psi - u$ cannot be minimal at $x_0$. This proves the claim. It follows that
\[
\int_{H\setminus B_R^+}W(\tilde u) \dx \leq 2\,\int_R^\infty e^{-2r}\,r^{n-1}\d r = P_n(R)\,e^{-2R}
\]
where $P_n$ is a polynomial of degree $n$ depending on the dimension. To estimate the second part of the energy functional, we use the gradient bound
\[
|\nabla u(x)| \leq n\,\sqrt{n}\,\sup_{\partial Q} |u| + \frac12\,\sup_Q |\Delta u|
\]
from \cite[Section 3.4]{gilbarg:2001vb} where $Q$ is a cube of side length $d=1$ with a corner at $x$. Applied to our problem, for $x\in \partial B_R^+$ we can find a cube $Q$ satisfying $\bar Q \cap \bar {B_R^+} = \{x\}$ such that
\[
|\nabla \tilde u(x)| \: =\: |\nabla (\tilde u-1)|(x) \leq n\,\sqrt{n}\,\sup_{\partial Q} |\tilde u -1 | + \frac12\,\sup_Q |W'(\tilde u)| \:\leq\: (n\,\sqrt{n} + 5/2)\,e^{-|x|}.
\]
Thus we also have
\[
\int_{H\setminus B_R^+}\frac12\,|\nabla \tilde u|^2\dx \:\leq\: \left(n\,\sqrt{n} + 5/2\right)^2\,P_n(R)\,e^{-2R}
\]
Finally, we remark that the same type of estimate obviously holds for $\Delta\tilde u = W'(\tilde u) \in L^2(H)$. Having given the general construction for suitable functions of zero $\W_1$ curvature energy, we are finally ready to apply these results to obtain counterexamples. For simplicity, we construct the counterexamples first on the half space $H$ and transfer them to bounded $\Omega$ later on.

\begin{example}[Counterexample to Boundedness]\label{counterexample 1}
Fix $h\in C_c^\infty(\R^n)$ such that $0\leq h\leq e^{-|x|}$, $h\not\equiv 0$ and set $h_\theta = \theta\,h$. Every function of this type induces a minimiser $\tilde u_\theta$. We may take a sequence $\theta_\eps\to \infty$ such that $\eps^{n-1}/\theta_\eps^{\,4} \to 0$ and set $u_\eps(x) = \tilde u_{\theta_\eps}(x/\eps)$. Clearly, $u_\eps$ becomes unbounded as $\eps\to 0$, but

\begin{enumerate}
\item $\W_\eps(u_\eps)\equiv 0$ and
\item $S_\eps(u_\eps) = \eps^{n-1}\,\F(\tilde u_{\theta_\eps}) \leq C\,\eps^{n-1}\,\F(h_{\theta_\eps})\to 0$.
\end{enumerate}
So the sequence $u_\eps$ induces limiting measures $\mu = \alpha = 0$, but fails to be uniformly bounded.
\end{example}

The  next example is a technically more demanding version of this one where the energy scaling is chosen so that we create an atom of size $S>0$ at the origin. 

\begin{example}[Counterexample to Boundary Regularity of $\mu$]
Take $h_\theta, \tilde u_\theta$ as above. Then the map
\[
f\colon [0,\infty) \to \R, \quad f(\theta) = \F(\tilde u_{\theta}) = \inf \{\F(u)\:|\:u\in 1+ h_{\theta} + W^{1,2}_0(H)\}
\]
is continuous. To see this, take pairs $\theta_1$, $\theta_2$ and the corresponding minimisers $\tilde u_1$, $\tilde u_2$ and observe that
\[
\tilde u_{1,2} = \frac{\theta_2}{\theta_1}\, \left[\tilde u_1 - 1\right] + 1 \:\:\in \:1 + h_{\theta_2} + W^{1,2}_0(H).
\]
Since
\[
W(1+\alpha u) = ((1+\alpha u)^2 -1)^2 /4 = (2\alpha u + \alpha^2 u^2)^2/4 \leq \max\{\alpha^2, \alpha^4\} W(1+u)
\]
we have
\[
f(\theta_2) = \F(\tilde u_2) \leq \F(\tilde u_{1,2}) \leq \max\left\{\left(\frac{\theta_2}{\theta_1}\right)^2, \:\left(\frac{\theta_2}{\theta_1}\right)^4\right\}\,\F(\tilde u_1) = \max\left\{\left(\frac{\theta_2}{\theta_1}\right)^2, \:\left(\frac{\theta_2}{\theta_1}\right)^4\right\}\,f(\theta_1).
\]
Reversing the roles of $\theta_1$ and $\theta_2$ shows that $f$ is continuous. Now let $S>0$. Due to the continuity of $f$ in $\theta$ and the trace inequality 
\[
\theta^2||h||_{2,\partial H}^2 = ||h_\theta||_{2,\partial H}^2 \leq \F(\tilde u_\theta)
\]
we can pick a sequence $\theta_\eps\to\infty$ at most polynomially in $1/\eps$ such that $\F(\tilde u_{\theta_\eps}) = S\,\eps^{1-n}$. As before, set $u_\eps(x) = \tilde u_{\theta_\eps}(x/\eps)$ and observe that $\W_\eps(u_\eps)\equiv 0$, $S_\eps(u_\eps)\equiv S$. It remains to show that $\mu = S\,\delta_0$, i.e.\ that the limiting measure is concentrated in one point. The functions $\tilde u_{\theta}$ actually tend to shift more of their mass towards the origin as $\theta\to \infty$ since the steepness (and overall height) is best concentrated on a ball of small radius for a low energy.

The same application of the maximum principle as before shows that $\tilde u_\theta \leq \tilde w_\theta \coloneqq  1 + \theta(\tilde u_1 -1)$ since 
\[
\Delta(\tilde w_\theta - \tilde u_\theta) = \theta\,\Delta \tilde u_1 - \Delta \tilde u_\theta = \theta\,W'(\tilde u_1) - W'(\tilde u_\theta)  \leq W'(\tilde w_\theta ) - W'(\tilde u_\theta)
\]
is monotone in $\tilde w_\theta$, $\tilde u_\theta$ and the boundary values satisfy $\tilde u_\theta = \tilde w_\theta$ on $\partial H$ and $\lim_{|x|\to\infty}\tilde u_\theta = \lim_{|x|\to\infty} \tilde w_\theta = 1$. Like above, we now obtain that
\[
\int_{H\setminus B_R^+}\frac12\,|\nabla \tilde u_\eps|^2 + W(\tilde u_\eps)\dx \leq \max\{\theta_\eps^2, \theta_\eps^4\}\,P_n(R)\,e^{-2R}.
\]
Thus we can choose a sequence $R_\eps\to \infty$ such that $\theta_\eps^4\, P_n(R_\eps)\,e^{-2R_\eps}\to 0$ and $\eps\,R_\eps\to 0$ since $\theta_\eps$ grows only polynomially in $1/\eps$ and the exponential term dominates (take e.g.\ $R_\eps = \eps^{-1/2}$). Thus for all $R>0$
\[
\mu_\eps(B_R(0)) = \eps^{1-n}\int_{B_{R/\eps}^+} |\nabla\tilde u_{\theta_\eps}|^2 + W(\tilde u_{\theta_\eps})\dx \geq \eps^{1-n} \int_{B_{R_\eps}^+} |\nabla\tilde u_{\theta_\eps}|^2 + W(\tilde u_{\theta_\eps})\dx \to S
\]
and hence $\mu(B_R(0))\geq S$. Taking $R\to 0$ shows that $\mu(\{0\}) = \mu(\overline H) = S$, i.e.\ $\mu = S\,\delta_0$. 
\end{example}

Functions as described above can appear as minimisers of functionals like $\W_\eps + \eps^{-1}\,(S_\eps - S)^2$ which are used to search for minimisers of Willmore's energy with prescribed surface area -- even as functions with energy zero. The same is true for functionals including the topological penalisation term discussed below. 

By construction, the previous example shows that the inclusion $\spt(\mu)\subset \lim_{\eps\to 0}u_\eps^{-1}(I)$ need not be true for any $I\cc (-1,1)$ since $u_\eps\geq 1$ and thus $K=\emptyset$. We use a similar construction to demonstrate that the reverse inclusion need not hold, either.

\begin{example}[Counterexample to Hausdorff Convergence]\label{counterexample 2}
Using the same arguments as above, if $0\leq h\leq 2$, we can find a solution $\tilde u \in (1-h) + W^{1,2}_0(H) \cap C^\infty_{loc}(\overline H)$ of 
\[
-\Delta \tilde u + W'(\tilde u) = 0\quad\text{in }H, \qquad \bar u = 1-h\quad\text{on }\partial H
\]
satisfying $-1\leq \tilde u\leq 1$, $\lim_{|x|\to\infty} \tilde u(x) = 1$ and $\F(\tilde u) \leq \F(1+h)<\infty$. Decay estimates are harder to obtain here since $W'$ is not monotone inside $[-1,1]$, but we will not need them, either. If we take $h$ such that $h(0) =2$, $h\in C_c^\infty(B_1)$, we can use continuity up to the boundary to deduce that $\tilde u\inv(\rho)\cap B_1^+ \neq \emptyset$ for all $\rho\in(-1,1)$. So when we set $u_\eps(x) = \tilde u(x/\eps)$, we see that

\begin{enumerate}
\item $\mu_\eps(H) = \eps^{n-1}\,\tilde \mu(H) = \eps^{n-1}\,\F(\tilde u)\to 0$,

\item $\W_\eps(u_\eps)\equiv 0$ and

\item $0\in \lim_{\eps\to 0}u_\eps^{-1}(I)$ in the Hausdorff sense for all $\emptyset\neq I\cc (-1,1)$.
\end{enumerate}
\end{example}

\begin{example}[Counterexample to Uniform H\"older Continuity]\label{counterexample 3}
If we take $h$ like in the previous example and replace it by $h^\omega(x)= h(\omega x)$ we observe that the associated minimisers satisfy
\[
\F(\tilde u^\omega) \leq \F(h^\omega) \leq \F(h)
\]
for all $\omega\geq 1$ since the gradient term stays invariant in two dimensions and decreases in three, while the integral of the double well potential decreases in both cases for any fixed $h$. Thus, if we take any sequence $\omega_\eps\to \infty$ and define $u_\eps(x) = \tilde u^{\omega_\eps}(x/\eps)$, we get the same results as before. As the function becomes steeper and steeper on the boundary faster than $\eps$, uniform H\"older continuity up to the boundary cannot hold, even for uniformly bounded boundary values.
\end{example}

\begin{example}[Counterexample to Boundary Regularity of $\mu$ with $-1<u_\eps<1$]
We can refine the examples to show that growth of $u_\eps$ on $\partial\Omega$ is not the only reason that $\mu$ might develop atoms on $\partial\Omega$, but that this is in fact possible with $|u_\eps|\leq 1$. This happens when we prescribe highly oscillating boundary values on $\partial H$. Let $h\in C_c^\infty(\partial H)$, then for any $u\in H^1(H)$ with $u|_{\partial H}=g$ we have 
\[
\int_H|\nabla u|^2\dx \geq [h]_{H^{1/2}(\partial H)}^2 = c_{n-1} \int_{\partial H\times \partial H}\frac{|h(x)-h(y)|^2}{|x-y|^{n+1}}\dx\dy.
\]
for a constant depending on the dimension $n-1\in\{1,2\}$. For any $S'>0$ and $\delta>0$ we can construct $h\in C^\infty (H)$ such that 
\begin{enumerate}
\item $0\leq h\leq\delta$,
\item $\supp(h) \subset B_1(0)$ and
\item $[h]_{H^{1/2}}^2 \geq S'$.
\end{enumerate}
We construct a solution of the stationary Allen-Cahn equation with the boundary values $1-h$ as before, but for a modified potential
\[
\ol W(s) = \begin{cases}W(1-2\delta) &s\leq 1-2\delta\\
	W(s) &s\geq 1-2\delta\end{cases}.
\]
An energy minimiser will never dip below $1-2\delta$ then, and consequently never below $1-\delta$ by the maximum principle if $\delta$ is chosen so small that $W'$ is monotone on $[1-2\delta, \infty)$. The rest of the proof goes through as before with suitable scaling of $h$ to get the right energy since $W'$ behaves correctly just below $1$, as it does slightly above $1$. We will not repeat the details.

The boundary values need to be constructed with slightly more care since we cannot just have vertical growth and the $H^{1/2}$-norm behaves badly under spacial scaling. This is compensated in the boundary construction by having a larger number of faster oscillations. When we have constructed $h$ with a large enough half-norm, we can always reduce it by scaling with a constant $<1$.
\end{example}
  
For the sake of simplicity, we chose to construct the examples on half space due to its scaling invariance. Let us sketch how they can be transferred to $C^2$-domains. If $\Omega\cc\R^n$ and $\partial\Omega\in C^2$ there exists $x_0\in \partial\Omega$ such that $|x_0| = \max_{x\in\partial\Omega}|x|$. At $x_0$, both principal curvatures of $\partial\Omega$ are strictly positive, so in a ball around $x_0$, up to a rigid motion we may write
\[
\Omega\cap B_r(x_0) = \{x\in B_r(x_0)\:|\:x_n>\phi(\hat x)\}
\]
where $\hat x = (x^1,\dots,x^{n-1})$ and $\phi$ is a strictly convex $C^2$-function satisfying $\phi(0)=0$, $\nabla \phi(0) = 0$ and $\Omega\subset H$. If $\Omega$ is convex in the first place, this is possible at every point $x_0\in \partial\Omega$. 

Thus, the function $u_\eps(x) = \tilde u(x/\eps)$ is well-defined on $\Omega$ for any of the functions $\tilde u$ constructed above. If $\eps$ is chosen small enough, the difference between $H$ and $\Omega/\eps$ becomes negligible for any given $\tilde u$ and we can still construct counterexamples to boundedness, local H\"older-continuity, relationship between $\spt(\mu)$ and the Hausdorff limit of the level sets and to the regularity of $\mu$ this way.

Using the exponential decay (or modifying functions to become constant for larger arguments) it is also possible to create singular behaviour for example along curves in the convex portion of the boundary by placing singular solutions of the stationary Allen-Cahn equation at an increasing number of points distributed along the curve.

We restricted our analysis to convex boundary points since then $u_\eps = \tilde u_\theta(x/\eps)$ is well-defined for all small $\eps>0$, whereas at other points, half space does not provide enough information to fill an entire neighbourhood of $x_0$. We believe that the same pathologies can arise at general boundary points.

\end{document}